\documentclass[12pt, reqno]{amsart}
\usepackage{amsmath, amsthm, amscd, amsfonts, amssymb, graphicx, color}
\usepackage[bookmarksnumbered, colorlinks, plainpages]{hyperref}
\hypersetup{colorlinks=true,linkcolor=red, anchorcolor=green, citecolor=cyan, urlcolor=red, filecolor=magenta, pdftoolbar=true}

\newtheorem{theorem}{Theorem}[section]
\newtheorem{lemma}[theorem]{Lemma}
\newtheorem{proposition}[theorem]{Proposition}
\newtheorem{corollary}[theorem]{Corollary}
\theoremstyle{definition}

\newtheorem{example}[theorem]{Example}

\theoremstyle{remark}
\newtheorem{remark}[theorem]{Remark}
\numberwithin{equation}{section}

\newcommand{\wbar}[1]{\overline{#1}}
\newcommand{\what}[1]{\widehat{#1}}
\newcommand{\wtil}[1]{\widetilde{#1}}
\newcommand{\til}[1]{\tilde{#1}}

\newcommand{\id}{\operatorname{id}}

\newcommand{\supp}{\operatorname{supp}}
\newcommand{\sgn}{\operatorname{sgn}}

\newcommand{\fC}{\mathcal{C}}
\newcommand{\fK}{\mathcal{K}}

\newcommand{\fM}{\mathcal{M}}

\newcommand{\Cee}{\mathbb{C}}
\newcommand{\Ree}{\mathbb{R}}
\newcommand{\Tee}{\mathbb{T}}
\newcommand{\Zee}{\mathbb{Z}}
\newcommand{\En}{\mathbb{N}}

\newcommand{\alp}{\alpha}
\newcommand{\del}{\delta}
\newcommand{\Del}{\Delta}
\newcommand{\eps}{\varepsilon}

\newcommand{\lam}{\lambda}

\newcommand{\sig}{\sigma}

\newcommand{\vphi}{\varphi}

\newcommand{\meas}{\mathrm{M}}
\newcommand{\blone}{\mathrm{L}^1}
\newcommand{\trig}{\mathrm{Trig}}

\newcommand{\norm}[1]{\left\Vert#1\right\Vert}

\begin{document}
\setcounter{page}{1}

\title[Commuting contractive idempotents]{Commuting contractive idempotents in measure algebras}

\author[N. Spronk]{Nico Spronk$^1$$^{*}$}

\address{$^{1}$ Department of Pure Mathematics, University of Waterloo,
Waterloo, ON, N2L 3G1, Canada}
\email{\textcolor[rgb]{0.00,0.00,0.84}{nspronk@uwaterloo.ca}}

\dedicatory{In honour of Tony's contributions to, and leadership in,
the international abstract harmonic analysis community,
the Canadian mathematical community,
and my career.}



\date{Received: xxxxxx; Revised: yyyyyy; Accepted: zzzzzz.
\newline \indent $^{*}$ Corresponding author}

\begin{abstract}
We determine when contractive idempotents in the measure algebra of a locally compact
group commute.  We consider a dynamical version of the same result.  We also look at
some properties of groups of measures whose identity is a contactive idempotent.
\end{abstract} \maketitle

Let $G$ be a locally compact group.  When $G$ is abelian, Cohen \cite{cohen}
characterised all of the idempotents in the measure algebra $\meas(G)$.  For non-abelian
$G$, the idempotent probabilities were characterized by Kawada and It\^{o} \cite{kawadaito},
while the contractive idempotents were characterized by Greenleaf \cite{greenleaf}.  We
give an exact statement of their results in Theorem \ref{theo:greenleaf}, below.
For certain compact groups, the central idempotent measures were characterized by Rider \cite{rider},
in a manner which is pleasingly reminiscent of Cohen's result on abelian groups.  Rider
points out a counterexample to his result when some assumptions are dropped.
This has motivated our Example \ref{ex:rider} (i), below.

Discussion of contactive idempotents has been conducted in the setting of locally compact
quantum groups by Neufang, Salmi, Skalski and the present author \cite{neufangsss}.

Under certain assumptions, results of Stromberg \cite{stromberg} and
Muhkerjea \cite{mukherjea}, show that convolution powers of a probability measure converge
either to an idempotent, or to $0$.  See Theorem \ref{theo:stromberg}
and Remark \ref{rem:mukherjea}, below.  We study limits of convolution powers
of products of contractive idempotents whose supports generate a compact subgroup.

We close with a study of certain groups of measures identified by Greenleaf \cite{greenleaf} and 
Stokke \cite{stokke} whose identities are contractive idempotents.  

\subsection{Notation and background}
We shall always let $G$ denote a locally
compact group with measure algebra $\meas(G)$.  We let
$\fK(G)$ denote the collection of all compact subgroups
of $G$.  For $K$ in $\fK(G)$ we let $m_K$ denote the normalised Haar measure on
$K$ as an element of $\meas(G)$.  
We shall identify the group algebra $\blone(K)$ as a subalgebra
of $\meas(G)$ via the identification $f\mapsto fm_K$, i.e.\ for $u\in\fC_0(G)$
we define
\[
\int_Gu\,d(fm_K)=\int_Ku(k)f(k)dk
\]
where $dk=dm_K(k)$. We let for $K$ in $\fK(G)$, $\what{K}_1$ denote the space
of multiplicative characters on $K$.  Hence $\what{K}_1$ is the dual group
of $K/[K,K]$, where $[K,K]$ is the closed commutator subgroup.  If $K$ is ableian
we will write $\what{K}$ for $\what{K}_1$.

Let us recall what is known about contractive idempotents.

\begin{theorem}\label{theo:greenleaf}
{\bf (i)} {\rm (Kawada and It\^{o} \cite{kawadaito})} If $\mu$ in $\meas(G)$ is a probability
with $\mu\ast\mu=\mu$, then there is $K$ in $\fK(G)$ with $\mu=m_K$.

{\bf (ii)} {\rm (Greenleaf \cite{greenleaf})}  If $\mu$ in $\meas(G)$ is a non-zero and contractive, 
$\norm{\mu}\leq 1$, and $\mu\ast\mu=\mu$, 
then there is $K$ in $\fK(G)$ and $\rho$ in $\what{K}_1$ for which $\mu=\rho m_K$.
\end{theorem}

Observe that all measures above are self-adjoint:
\[
\int_G u\, d(\rho m_K)^*=\int_K u(k^{-1})\wbar{\rho(k)}\,dk=\int_K u(k)\rho(k)\,dk
=\int_G u\,d(\rho m_K)
\]
thanks to unimodularity of the compact group $K$.

\section{Main Result}

In order to proceed, let us consider some conditions under which products of groups
are groups.

\begin{lemma}\label{lem:subgrprel}
Let $K_1,K_2\in\fK(G)$.  Then the folowing are equivalent

{\bf (i)} $K_1K_2=\{k_1k_2:k_1\in K_1,k_2\in K_2\}\in\fK(G)$, 

{\bf (ii)} $K_1K_2$ is closed under inversion, and

{\bf (ii)} $K_1K_2=K_2K_1$.
\end{lemma}

\begin{proof} Note first that $K_1K_2$ is always a compact subset of $G$
which contains the identity $e$.  If (i) holds, then (ii) holds.  
We have that $(K_1K_2)^{-1}=K_2K_1$, which immediately
shows the equivalence of (ii) and (iii).    Finally if (iii) holds
then it is clear that $K_1K_2$ is closed under multiplication.
Thus, since (iii) implies (ii), we see that $K_1K_2$ is closed
under multiplication and inversion, hence we obtain (i).  \end{proof}

We observe that $K_1K_2\in\fK(G)$ in the following situations:

{\bf (i)} $K_1\subset K_2$, and

{\bf (ii)} $K_1\subset N_G(K_2)=\{s\in G:sK_2s^{-1}=K_2\}$.

\noindent 
If $K_1\cap K_2=\{e\}$ and $K_1K_2\in\fK(G)$, then
$(K_1,K_2)$ is referred to as a {\em matched pair} \cite{takeuchi},
and $K_1K_2$ is a {\em Zappa-Sz\'{e}p product} \cite{zappa,szep}.
Indeed, we note that
the representation $k_1k_2$ of an element of $K_1K_2$ is unique
for if $k_1k_2=k_1'k_2'$, then $(k_1')^{-1}k_1=k_2'k_2^{-1}=e$.
Since, in general, we will not assume that $K_1\cap K_2=\{e\}$,
nor even that this intersection is normal in $K_1K_2$, when the latter is a group,
our situation appears to generalize that of a matched pair.

Is there a ``nice" characterization of when $K_1K_2\in\fK(G)$?

To proceed we shall use a non-normal form of the Weyl integration formula.  
If $H$ is a locally compact group and $L\in\fK(H)$,
then any continuous multiplicative function $\del:L\to \Ree^{>0}$
is trivial.  Thus the modular function $\Del$ of $H$ satisfies $\Del|_L=1$,
which is the modular function of $L$.  Hence the left homogeneous space
$H/L$ admits a left $H$-invariant Haar measue $m_{H/L}$.  We have
for $u$ in $\fC_c(H)$ that
\begin{equation}\label{eq:weyl}
\int_H u(h)\,dh =\int_{H/L}\int_L u(hl)\, dl\, d(hL)
\end{equation}
where $d(hL)=dm_{H/L}(hL)$.

\begin{theorem}\label{theo:contcommute}
Let $K_1,K_2\in\fK(G)$, $\rho_1\in\what{(K_1)}_1$ and
$\rho_2\in\what{(K_2)}_1$.  Then $\rho_1m_{K_1}$ and
$\rho_2m_{K_2}$ commute if and only if one of the following
cases holds for $K=K_1\cap K_2$:

{\bf (i)} $\rho_1|_K\not=\rho_2|_K$, in which case
$(\rho_1m_{K_1})\ast(\rho_2m_{K_2})=0$; or

{\bf (ii)} $\rho_1|_K=\rho_2|_K$, $K_1K_2\in\fK(G)$ and
the function 
\[
\rho:K_1K_2\to\Cee\text{ given by }\rho(k_1k_2)=\rho_1(k_1)\rho_2(k_2)\text{ for }
k_1\text{ in }K_1\text{ and }k_2\text{ in } K_2 
\]
defines a character; in which case $(\rho_1m_{K_1})\ast(\rho_2m_{K_2})=\rho m_{K_1K_2}$.

In particular, the idempotent probabilities $m_{K_1}$ and $m_{K_2}$ commute
if and only if $K_1K_2\in\fK(G)$, and
we have $m_{K_1}\ast m_{K_2}=m_{K_1K_2}$, in this case.  
\end{theorem}

\begin{proof}
We let $\nu=(\rho_1m_{K_1})\ast(\rho_2m_{K_2})$.  Notice that 
\begin{equation}\label{eq:nustar}
\rho_1m_{K_1}\text{ and }\rho_2m_{K_2}\text{ commute if and only if }\nu^*=\nu.
\end{equation}
For $u$ in $\fC_0(G)$ we have
\begin{align}
\int_G u\,d\nu
&=\int_{K_1}\int_{K_2}u(k_1k_2)\rho_1(k_1)\rho(k_2) \,dk_1\,dk_2  \notag \\
&=\int_{K_1/K}\int_K\int_{K_2}u(k_1kk_2)\rho_1(k_1k)\rho(k_2)\,dk_2
\,dk\,d(k_1K)  \notag \\
&=\int_{K_1/K}\int_{K_2}u(k_1k_2)\int_K\rho_1(k_1k)\rho_2(k^{-1}k_2)\,dk
\,dk_2\, d(k_1K)  \label{eq:computation1} \\
&=\int_{K_1/K}\int_{K_2}\left[\int_K\rho_1(k)\wbar{\rho_2(k)}\,dk\right]
u(k_1k_2)\rho_1(k_1)\rho_2(k_2)\,dk_2\, d(k_1K). \notag
\end{align}
The orthogonality of characters entails that the
quantity $\int_K\rho_1(k)\wbar{\rho_2(k)}\,dk$ is either $1$ or $0$, depending on
whether $\rho_1|_K=\rho_2|_K$ or not.  In the latter case, we see that $\nu=0$, and
hence $(\rho_2m_{K_2})\ast(\rho_1m_{K_1})=\nu^*=0=\nu$, and we see that condition (i) holds.

Hence for the remainder of the proof,
let us suppose that $\rho_1|_K=\rho_2|_K$.  Then the function $\rho:K_1K_2\to\Tee$
given as in (ii) is well-defined.
Indeed, if $k_1k_2=k_1'k_2'$, then $(k_1')^{-1}k_1=k_2'k_2^{-1}\in K$,
and our assumption allows us to apply $\rho_1$ to the left, and $\rho_2$ 
to the right, to gain the same result.    Furthermore, 
$(k_1,k_2)\mapsto \rho_1(k_1)\rho_2(k_2)=\rho(k_1k_2):K_1\times K_2\to\Tee$
is continuous and hence factors continuously through the topological quotient space $K_1K_2$
of $K_1\times K_2$.  

We now wish to show that $\supp\nu=K_1K_2$.  The inclusion $\supp\nu\subseteq K_1K_2$
is standard.  Conversely, if  $k_1^o$ in $K_1$, $k_2^o$ in $K_2$ and
$\eps>0$ are given and let $u,v\in\fC_0(G)$ satisfy 
\[
u\geq 0\text{ and }u(k_1^ok_2^o)>\eps>0\text{; and }v|_{K_1K_2}=\bar{\rho}.
\]
Then we may find open $U_1$ containing $k_1^o$ and open $U_2$ containing $k_2^o$
so that $U_1\times U_2\subseteq\{(k_1,k_2)\in K_1\times K_2:u(k_1k_2)>\eps\}$, and
our assumptions entail that
\begin{align*}
\int_G uv\,d\nu&=\int_{K_1}\int_{K_2}u(k_1k_2)\,dk_1\,dk_2 \\
&\geq\int_{U_1}\int_{U_2}u(k_1k_2)\,dk_1\,dk_2\geq
m_{K_1}(U_1)m_{K_2}(U_2)\eps>0.
\end{align*}
Hence $K_1K_2\subseteq\supp\nu$.  Notice that if it were the case that $\nu=0$, this would
contradict our present calculation of $\supp\nu$, and hence the assumption that 
$\rho_1|_K=\rho_2|_K$.  Thus $\nu=0$ only when $\rho_1|_K\not=\rho_2|_K$, showing
that (i) fully characterizes this situation.  We observe that 
\begin{equation}\label{eq:nustar2}
K_1K_2=\supp\nu^*=(\supp\nu)^{-1}=K_2K_1.
\end{equation}

Let us now assume  that $\rho_1m_{K_1}$ and $\rho_2m_{K_2}$ commute. Then, by
(\ref{eq:nustar}), $\nu=\nu^*$ and hence by (\ref{eq:nustar2}) and
Lemma \ref{lem:subgrprel}, we have that $K_1K_2\in\fK(G)$.
To complete the calculation we observe the following isomorphism of left $K_1$-spaces,
generalizing the second isomorphism theorem of groups:
\begin{equation}\label{eq:konespace}
K_1K_2/K_2\cong K_1/K,\; kK_2\mapsto kK.
\end{equation}
Hence for $u\in\fC(K_1K_2)$ which is constant of left cosets of $K_2$ we have 
\linebreak 
$\int_{K_1K_2/K_2}u(k)\, d(kK_2)=\int_{K_1/K} u(k_1)\,
d(k_1K)$, for the unique choices of left-invariant probability measures
on the homogeneous spaces.  We thus find that
\begin{align}
\int_G u\,d\nu&=\int_{K_1/K}\int_{K_2}u(k_1k_2)\rho(k_1k_2)\,dk_2\, d(k_1K) \notag \\
&=\int_{K_1K_2/K_2}\int_{K_2}u(k_1k_2)\rho(k_1k_2)\,dk_2\, d(kK_2) \label{eq:computation2} \\
&=\int_{K_1K_2} u(k)\rho(k)\, dk=\int_G u\,d(\rho m_{K_1K_2}) \notag
\end{align}
so $\nu=\rho m_{K_1K_2}$.
Since $\nu\ast\nu=\nu$, as $\rho_1m_{K_1}$ and $\rho_2m_{K_2}$ commute,
and $m_{K_1K_2}$ is the normalized
Haar measure of a compact subgroup, it follows that $(\rho m_{K_1K_2})
\ast(\rho m_{K_1K_2})=(\rho\ast\rho) m_{K_1K_2}$, whence $\rho=\rho\ast\rho$.
We could appeal immediately to Theorem \ref{theo:greenleaf} (ii), to see that
since $\|\nu\|\leq\|\rho_1m_{K_1}\|\|\rho_2 m_{K_2}\|=1$, that
$\rho\in\what{(K_1K_2)}_1$.  However,  let us give a direct verification, using only the present tools.
We may interchange the roles
of $K_1$ and $K_2$ above, and define $\til{\rho}:K_2K_1\to\Tee$  by
$\til{\rho}(k_2k_1)=\rho_2(k_2)\rho_1(k_1)$, which, 
like $\rho$, is well-defined and continuous.
We also see, by the computation (\ref{eq:computation2}), 
that $\nu=\til{\rho}m_{K_2K_1}=\til{\rho}m_{K_1K_2}$.  Hence
$\til{\rho}=\rho$ on $K_1K_2$.  But it then follows that $\rho$ is a homomorphism:
if $k=k_1k_2$, $l=l_1l_2$, $k_1,l_1\in K_1$, $k_2,l_2\in K_2$, we have
$k_2l_1=l_1'k_2'$ for some $l_1'$ in $K_1$ and $k_2'$ in $K_2$ and hence
\begin{align*}
\rho&(k_1k_2l_1l_2)=\rho_1(k_1l_1')\rho_2(k_2'l_2)=\rho_1(k_1)\rho(l_1'k_2')\rho_2(l_2) \\
&=\rho_1(k_1)\til{\rho}(k_2l_1)\rho_2(l_2)=\rho_1(k_1)\rho_2(k_2)\rho_1(l_1)\rho_2(l_2)
=\rho(k_1k_2)\rho(l_1l_2).
\end{align*}

Conversely, if the conditions of (ii) are assumed, then computations (\ref{eq:computation1}) and
(\ref{eq:computation2}) show that $(\rho_1 m_{K_1})\ast(\rho_2 m_{K_2})=\rho m_{K_1K_2}$
and show the same with the roles of $\rho_1 m_{K_1}$ and $\rho_2 m_{K_2}$, reversed.
\end{proof}

\begin{example}\label{ex:rider}
{\bf (i)} Let $G=K\rtimes A$ where $A$ is a compact group acting as continuous 
automorphisms on the group $K$, so we obtain
group law $(k,\alp)(k',\beta)=(k\alp(k'),\alp\beta)$.  
We identify $K$ and $A$ with their
cannonical copies in $G$ and suppose there is
$\rho$ in $\what{K}_1$ for which $\rho\circ\alp\not=\rho$ for some $\alp$ in $ A$,
and hence for $\alp$ on an open subset of $A$. (A specific example
would be to take $K=\Tee$, $A=\{\id,\sigma\}$ where $\sigma(t)=t^{-1}$,
and $\rho(t)=t^n$ where $n\in\Zee\setminus\{0\}$.)
Then for $u\in\fC(G)$ we obtain for $\rho$ as above
\[
\int_G u\, d[(\rho m_K)\ast m_A]=\int_K \int_A u(k,\alp)\rho(k)\, d\alp\,dk
\]
while, since the modular function on the compact group $A$ qua automorphisms on $K$ is $1$, 
we have 
\begin{align*}
\int_G u\, d[m_A\ast(\rho m_K)]
&=\int_A\int_K u(\alp(k),\alp)\rho(k)\,dk\,d\alp \\
&=\int_A\int_K u(k,\alp)\rho\circ\alp^{-1}(k)\,dk\,d\alp.
\end{align*}
Thus $\rho m_K$ and $m_A$ do not commute.
The only assumption missing from Theorem
\ref{theo:contcommute} is that $(k,\alp)\mapsto \rho(k)$ is
not a character on $G$.

{\bf (ii)}  Let $n\geq 5$ and $S_n$ the symetric group on a set of $n$ elements,
let $S_{n-1}$ denote the stabiliser subgroup of any fixed element, and $C$ the cyclic
subgroup generated by any full $n$-cycle.  Then $S_n=S_{n-1}C$, as may be easily checked,
and $\{S_{n-1},C\}$ is a ``non-trivial" matched pair
in the sense that neither subgroup is normal in $G$.

We note that the only non-trivial co-abelian normal subgroup of $S_n$ is $A_n=\ker\sgn$,
as $A_n$ is simple and of index $2$;
hence $\what{(S_n)}_1=\{1,\sgn\}$.  Hence if $\rho_2$ in $\what{C}\setminus\{1\}$ satisfies
$\rho_2\not=\sgn|_C$, then for any $\rho_1$ in $\what{(S_{n-1})}_1$, it follows from
Theorem \ref{theo:contcommute} that
$(\rho_1 m_{S_{n-1}})\ast(\rho_2 m_C)\not=(\rho_2 m_C)\ast(\rho_1 m_{S_{n-1}})$.
\end{example}

\section{Dynamical considerations}

If $S$ is a subset of $G$,  let $\langle S\rangle$ denote the smallest closed
subgroup containing $S$.

\begin{theorem}\label{theo:stromberg}
{\rm (Stromberg \cite{stromberg})}
If $\mu$ is a probability in $\meas(G)$, for which
$K=\langle\supp\mu\rangle\in\fK(G)$,  then 
the weak* limit, $\lim_{n\to\infty}\mu^{\ast n}$, exists if and only if
$\supp\mu$ is contained in no coset of a closed proper normal subgroup of $K$.
Moroever, this limit equals the Haar measure $m_K$.
\end{theorem}

We observew that $\supp\mu^*=(\supp\mu)^{-1}$, and hence in the assumptions above we have
$\lim_{n\to\infty}(\mu^*)^{\ast n}=m_K$ too.

Since $\supp(m_K\ast m_L)=KL$, as was checked in the proof of Theorem \ref{theo:contcommute},
it follows that for $K,L$ in $\fK(G)$ for which $\langle KL\rangle$
is compact, we have $\lim_{n\to\infty}(m_K\ast m_L)^{\ast n}=m_{\langle KL\rangle}
=\lim_{n\to\infty}(m_L\ast m_K)^{\ast n}$.  For example, in
$S=\mathrm{SU}(2)$, any two distinct (maximal) tori $T_1$ and $T_2$ generate
$S$, as the only subgroups of $S$ with non-trivial connected components are
tori, or $S$, itself.  Hence $m_S=\lim_{n\to\infty}(m_{T_1}\ast m_{T_2})^{\ast n}$.

Futhermore,  we can deduce from the observation above that $m_L$ and $m_K$ commute
if and only if $KL=\langle KL\rangle$, giving the special case of Theorem \ref{theo:contcommute}.

Motivated by the above considerations, we consider the following dynamical result.  

\begin{theorem}\label{theo:dynamical}
Let $K_j\in\fK(G)$ and $\rho_j\in\what{(K_j)}_1$ for $j=1,\dots,m$ for which
$L=\langle K_1\dots K_m\rangle\in\fK(G)$.  
Then the weak* limit
\[
\lim_{n\to\infty}[(\rho_1 m_{K_1})\ast\dots\ast(\rho_m m_{K_m})]^{\ast n}
\]
always exists.  It is $\rho m_L$, provided there is a $\rho$ in $\what{L}_1$ for which
$\rho|_{K_j}=\rho_j$ for each $j$, and $0$ otherwise.
\end{theorem}

\begin{proof}
We let $\nu=(\rho_1 m_{K_1})\ast\dots\ast(\rho_m m_{K_m})$.  Then
each $\nu^{\ast n}$, being a product of contractive elements, satisfies
$\|\nu^{\ast n}\|\leq 1$.  The Peter-Weyl theorem tells us that the algebra $\trig(L)$ consiting of
matrix coefficients of finite-dimensional unitary representations, is uniformly dense in 
in $\fC(L)$.  Hence, since $\supp\nu\subseteq L$ and $\|\nu\|\leq 1$, hence
$\|\nu^{\ast n}\|\leq 1$ for each $n$,  it suffices to determine,  for
any finite dimensional unitary representation unitary $\pi:L\to \mathrm{U}(d)$, the nature of
the limit
\begin{equation}\label{eq:pinu}
\lim_{n\to\infty}\pi(\nu^{\ast n})=\lim_{n\to\infty}\int_L \pi(l)\,d\nu^{\ast n}(l)\text{ in }M_d(\Cee).
\end{equation}
It is well-known, and simple to compute that each 
\[
\pi(\nu^{\ast n})=\pi(\nu)^n=[\pi(\rho_1 m_{K_1})\dots \pi(\rho_1 m_{K_1})]^n.
\]
For each $j=1,\dots,m$ the Schur orthogonality relations tell us that
\[
\pi(\rho_j m_{K_j})=\int_{K_j}\rho_j(k)\pi(k)\,dk=p_j
\]
where $p_j$ is the orthogonal projection onto the space of vectors $\xi$ for which
$\pi(k)\xi=\wbar{\rho_j(k)}\xi$ for each $k$ in $K_j$.  Hence it follows that
\[
\pi(\nu)=p_1\dots p_m\text{ and }\pi(\nu^{\ast n})=(p_1\dots p_n)^n.
\]
Since each $p_j$ is contractive, the eigenvalues of $\pi(\nu)$ are of modulus
not exceeding one.  Furthermore, if $\|\pi(\nu)\xi\|_2=\|\xi\|_2$ (Hilbertian norm), then
we find that
\[
\|\xi\|_2=\|p_1\dots p_m\xi\|_2\leq\|p_2\dots p_m\xi\|_2\leq\dots\leq\|p_m\xi\|_2\leq\|\xi\|_2
\]
so equality holds at each place.  But  we see then that $\xi$ is in the range of
$p_m$, hence of $p_{j-1}$ if it is in the range of $p_j$, and thus in the mutual range
$R_\pi$ of each of $p_1,\dots, p_m$.  If we consider the Jordan form of $\pi(\nu)
=p_1\dots p_m$, we see that $\lim_{n\to\infty}\pi(\nu)^n=q$, where $q$ is the
necessarily contractive, hence orthogonal, range projection onto $R_\pi$.  But then for
$\xi$ in $R_\pi$ and $k_j$ in $K_j$, $j=1,\dots,m$, we have
\[
\pi(k_1\dots k_n)\xi=\pi(k_1)\dots\pi(k_m)\xi
=\wbar{\rho_1(k_1)}\dots\wbar{\rho_n(k_n)}\xi.
\]
If we have $\xi\not=0$, then $\Cee\xi$ is $\pi(K_1\dots K_m)$-invariant, hence 
$\pi$-invariant as $L=\langle K_1\dots K_m\rangle$.  Moreover, there is, then,
$\rho$ in $\what{L}_1$ for which $\pi(l)\xi=\rho(l)\xi$, and it follows that
$\rho|_{K_j}=\rho_j$.  Notice that this $\rho$ is determined independently
of the choice of $\xi$, and hence even the choice of $\pi$.  In particular, if no
such $\rho$ exists, i.e.\ for every finite dimensional unitary representation $R_\pi=\{0\}$,
then we have $\lim_{n\to\infty}\nu^{\ast n}=0$, in the weak* sense.  When this $\rho$ does
exists, we see for
$u$ in $\fC_0(G)$ that each $\int_G u\,d(\nu^{\ast n})$ is given by
\begin{align}
\int_{K_1}&\dots\int_{K_m}\dots\dots\int_{K_1}\dots\int_{K_m}
u(k_{11}\dots k_{1m}\dots k_{n1}\dots k_{nm}) \notag \\
& \phantom{=n}\rho_1(k_{11})\dots\rho_m(k_{1m})
\dots \rho_1(k_{n1})\dots\rho_m(k_{nm})\,
dk_{nm}\dots dk_{n1}\dots dk_{1m}\dots dk_{11} \notag \\
&=\int_{K_1}\dots\int_{K_m}\dots\dots\int_{K_1}\dots\int_{K_m}
u(k_{11}\dots k_{nm})\rho(k_{11}\dots k_{nm})\,dk_{nm}\dots dk_{11}  \notag \\
&=\int_G u\rho\, d([m_{K_1}\ast\dots\ast m_{K_m}]^{\ast n}).\label{eq:mintgral}
\end{align}
It is easy to verify, as in the proof of Theorem \ref{theo:contcommute},
that $\sup(m_{K_1}\dots m_{K_m})=K_1\dots K_m$.  Hence
by Theorem \ref{theo:stromberg} we have obtain weak* limit
\[
\lim_{n\to\infty}\nu^{\ast n}=\rho m_L
\]
as desired.
\end{proof}

In fact, the above result  generalizes the necessity direction
of Theorem \ref{theo:contcommute}.

\begin{corollary}\label{cor:dynamical}
Let $K_j$ and $\rho_j$, $j=1,\dots,m$, be as in Theorem \ref{theo:dynamical},
above, and $L=K_1\dots K_m$.  If $\nu=(\rho_1 m_{K_1})\ast\dots\ast(\rho_m m_{K_m})$
is idempotent then either $\nu=0$, or $L=\langle L\rangle\in\fK(G)$ and there is 
$\rho$ in $\what{L}_1$ with $\rho|_{K_j}=\rho_j$ for each $j$.
\end{corollary}

\begin{proof}
Suppose $\nu\not=0$.
By a similar method as in the proof of  Theorem \ref{theo:contcommute}, we see that 
$\supp\nu=L$.  Moreover, if $\nu$ is idempotent, then $\lim_{n\to\infty}\nu^{\ast n}=\nu$.
Hence we obtain that $L=\langle L\rangle$, and there exists a
multiplicative character $\rho$ on $L$, as promised, thanks to Theorem \ref{theo:dynamical}.
\end{proof}

Though Corollary \ref{cor:dynamical}  generalizes the necessity direction of
Theorem \ref{theo:contcommute}, the proof of the earlier result is more self-contained, not relying
on Stromberg's result.  Furthermore, the sufficiency direction of Theorem \ref{theo:contcommute}
cannot be generalized so easily, even with probability idempotent measures.

\begin{example}\label{ex:sothree}
The special orthogonal group $S=\mathrm{SO}(3)$ admits the well-known Euler angle
decomposition: $S=T_1T_2T_1$ where
\begin{align*}
T_1&=\left\{k_1(t)=\begin{bmatrix} 1 & 0 & 0 \\ 0 & \cos t & -\sin t \\ 0 & \sin t & \cos t\end{bmatrix}:
0\leq t\leq 2\pi\right\}\text{ and } \\
T_2&=\left\{k_2(t)=\begin{bmatrix} \cos t & -\sin t & 0 \\ \sin t & \cos t & 0 \\ 0 & 0 & 1 \end{bmatrix}:
0\leq t\leq 2\pi\right\}.
\end{align*}
We note that multiplication $T_1\times (T_2/\{I,k_2(\pi)\})\times T_1\to S$ is a diffeomorphism.
For $u$ in $\fC(S)$ we have
\[
\int_{T_1}\int_{T_2}\int_{T_1}u\,d(m_{T_1}\ast m_{T_2}\ast m_{T_1})
=\int_0^{2\pi}\int_0^{2\pi}\int_0^{2\pi} u(k_1(t_1)k_2(t_2)k_1(t_3))\,\frac{dt_3\,dt_2\,dt_1}{8\pi^3}
\]
whereas the Haar measure $m_S$ gives integral
\[
\int_Su\,dm_S=
\int_0^{2\pi}\int_0^{\pi}\int_0^{2\pi} u(k_1(t_1)k_2(t_2)k_1(t_3))\sin t_2\,\frac{dt_3\,dt_2\,dt_1}{8\pi^2}.
\]
Hence considering $T_1$-spherical functions, i.e.\ $u$ in $\fC(T_1\backslash S/T_1)$,
we see that 
\[
m_{T_1T_2T_1}=m_S\not=m_{T_1}\ast m_{T_2}\ast m_{T_1}.
\]
\end{example}

\begin{remark}\label{rem:mukherjea}
We note the following result, shown (implicitly) by Muhkerjea \cite[Theo.\ 2]{mukherjea}.
{\it If $\mu$ is a probability in $\meas(G)$, for which
$\langle\supp\mu\rangle\not\in\fK(G)$,  then 
the weak* limit satisfies $\lim_{n\to\infty}\mu^{\ast n}=0$.}  
\end{remark}

Hence if $K_1,\dots,K_m$ in $\fK(G)$
have $\langle K_1\dots K_m\rangle\not\in\fK(G)$, we see that 
\linebreak 
$\lim_{n\to\infty}(m_{K_1}\ast \dots\ast m_{K_m})^{\ast n}=0$, which is rather antithetical to 
having $m_{K_1}\ast \dots\ast m_{K_m}$ be an idempotent.

As a simple example, consider the any two non-trivial finite subgroups
$K$ and $L$ of discrete groups $\Gamma$ and $\Lambda$, and consider
each as a subgroup of the free product $\Gamma\ast \Lambda$.
For a Lie theoretic example, consider the Iwasawa decomposition $KAN$ of 
$S=\mathrm{SL}_2(\Ree)$.  Compute that if $a\in A\setminus\{I\}$, then $aKa^{-1}\not=K$.
Since $K$ is maximal compact, we see that $\langle KaKa^{-1}\rangle\not\in\fK(S)$.

It is the case that if for $K_1,\dots,K_m$ in $\fK(G)$ we have 
$H=\langle K_1\dots K_m\rangle\not\in\fK(G)$,
then for $\rho_j$ in $\what{(K_j)}_1$, $j=1,\dots,m$, the weak* limit satisfies
\begin{equation}\label{eq:rn}
\lim_{n\to\infty}[ (\rho_1 m_{K_1} )\ast\dots\ast (\rho_m m_{K_m}) ]^{\ast n}=0.
\end{equation}
In the case that $\rho_i|_{K_i\cap K_j}\not=\rho_j|_{K_i\cap K_j}$ for some $i\not=j$,
we have $(\rho_1 m_{K_1})\ast\dots\ast(\rho_m m_{K_m})=0$, as may be computed, 
by a straightforward adaptation of (\ref{eq:computation1}).  If there is a
continuous multiplicative character $\rho:H\to\Tee$ such that $\rho|_{K_j}=\rho_j$ for each
$j$, then the computation (\ref{eq:mintgral}), and Mukherjea's theorem give the result.
In presence or absense of these assumptions,  (\ref{eq:rn}) follows from a result which
should appear in work of Neufang, Salmi, Skalski and the author, in progress.  
In fact, the same result implies Theorem \ref{theo:dynamical}.  
However, the proof given in the present note uses simpler methods.

\section{On groups of measures}

Geenleaf's motivation for studying idempotent measures was their use in the
study of contractive homomorphisms
$\blone(H)\to\meas(G)$.  In doing so, he required a description of certain groups of measures, given
in Theorem \ref{theo:groupsinmeas}, below.  We are  interested in determining
how these groups interact under convolution product with each other.  Stokke \cite{stokke}
conducted a study of Greenleaf's groups, and also devised a more general class of groups;
see (\ref{eq:stokke}).  We show that the latter class is indeed more general.

Let for any subgroup $H$ of $G$
\[
N_G(H)=\{g\in G:gHg^{-1}= H\}\text{ and }
Z_G(H)=\{g\in G:gh=hg\text{ for all }h\text{ in }H\}
\]
denote its normalizer and centralizer, respectively.  Notice that for another subgroup $L$, we have
$L\subseteq Z_G(H)$ if and only if $H\subseteq Z_G(L)$.  Notice too that for the topological closure
$\wbar{H}$, we have $N_G(\wbar{H})=N_G(H)$ and $Z_G(\wbar{H})=Z_G(H)$, and hence
these subgroups are closed.

Given $K$ in $\fK(G)$ and $\rho$ in $\what{K}_1$ we let
\[
N_{K,\rho}=N_G(K)\cap N_G(\ker\rho)
\]
and then let $q:N_{K,\rho}\to N_{K,\rho}/\ker\rho$ be the quotient map.  We let
\[
G_{K,\rho}=q^{-1}(Z_{N_{K,\rho}/\ker\rho}(K/\ker\rho)).
\]
Hence $g$ in $G_{K,\rho}$ normalizes both $K$ and $\ker\rho$, and commutes with elements of $K$
modulo $\ker\rho$.  We then consider, in $\meas(G)$, the subgroup
\[
\Gamma_{\rho m_K}
=\{z\del_g\ast(\rho m_K):z\in\Tee\text{ and }g\in G_{K,\rho}\}
\]
We remark that  $G_{K,\rho}=\{g\in G:\del_g\ast(\rho m_K)=(\rho m_K)\ast\del_g\}$, and
$\Gamma_{\rho m_K}$ is topological group with the weak*-topology
on $\meas(G)$ and multiplication $z\del_g\ast(\rho m_K)\ast z'\del_{g'}\ast(\rho m_K)
=zz'\del_{gg'}\ast(\rho m_K)$.

\begin{theorem}\label{theo:groupsinmeas}
{\bf (i)} {\rm (Greenleaf \cite{greenleaf})}  Any closed group of contractive measures has identity of the
form $\rho m_K$ of Theorem \ref{theo:greenleaf} (ii), and is a subgroup of $\Gamma_{\rho m_K}$.

{\bf (ii)} {\rm (Stokke \cite{stokke}, after \cite{greenleaf})}  The map
\[
(z,g)\mapsto z\del_g\ast(\rho m_K):\Tee\times G_{K,\rho}\to\Gamma_{\rho m_K}
\]
is continuous and open and with compact kernel $\{(\rho(k),k):k\in K\}\cong K$.
\end{theorem}

\begin{remark}\label{rem:stokke}
We give a mild simplification of Stokke's argument, which will help us, below.

{\bf (i)} {\it Let 
\[
\Omega_{K,\rho}
=(\Tee\times G_{K,\rho})/\{(\rho(k),k):k\in K\}.
\]
Then the one point compactification $\Omega_{K,\rho}\sqcup\{\infty\}$ (respectively,
topological coproduct, if $G_{K,\rho}$ is compact) is homeomorphic to 
$\Gamma_{\rho m_K}\cup\{0\}$.}

Indeed, consider the semigroup homomorphism on 
$(\Tee\times G_{K,\rho})\sqcup\{\infty\}$ given by
 $(z,g)\mapsto z\del_g\ast(\rho m_K)$, $\infty\mapsto 0$,
which has kernel $\{(\rho(k),k):k\in K\}$ at the identity
--- a fact which we shall take for granted, thanks to arguments in \cite{stokke,greenleaf}.
It suffices to verify that this semigroup homorphism is continuous
and that $\Gamma_{\rho m_K}\cup\{0\}$ is weak*-compact.  Let $(z_i,g_i)$ be a net in
$\Tee\times G_{K,\rho}$ such that $z_i\del_{g_i}\ast(\rho m_K)\to \mu$ in $i$.  If
$(g_i)$ is unbounded in $G_{K,\rho}$, we may pass to subnet and assume $g_i\to \infty$.
But then for $u$ in $\fC_0(G)$, $(u(g_i\cdot))$ converges to zero uniformly on compact sets,
thanks to uniform continuity of $u$.  It follows that $\mu=0$.  Otherwise $(g_i)$ is bounded
in $G_{K,\rho}$, and by passing to subnet, we may assume that $(z_i,g_i)\to(z,g)$
in $\Tee\times G_{K,\rho}$.  But then  for $u$ in $\fC_0(G)$, $(u(g_i\cdot))$ converges to
$u(g\cdot)$ uniformly on compact sets, and it follows that $\mu=z\del_g\ast(\rho m_K)$.
Notice that any limit point of a net in $\Gamma_{\rho m_K}$ is in $\Gamma_{\rho m_K}\cup\{0\}$,
so the latter set is weak*-closed, hence weak*-compact as it is a subset of the weak*-compact
unit ball of $\meas(G)$.

{\bf (ii)}  {\it If $H$ is any closed subgroup of $G_{K,\rho}$, then 
\[
\bigl((\Tee\times H)/\{(\rho(k),k):k\in K\cap H\}\bigr)\sqcup\{\infty\}
\]
is homoemorphic to $\{z\del_g\ast(\rho m_K):z\in\Tee\text{ and }g\in H\}\cup\{0\}$.
Moroever, the latter set is weak*-compact.}
These facts are immediate from (i), above.
\end{remark}

For a set $\Sigma$ of contractive measures, let $[\Sigma]$ denote the 
smallest weak*-closed semigroup containing $\Sigma$.

\begin{proposition}\label{prop:groups}
Suppose $K_1$, $K_2$, $\rho_1$ and $\rho_2$ satisfy the contditions of
Theorem \ref{theo:contcommute} (ii), and let $\rho$ be as given there.
Then
\[
\left[\Gamma_{\rho_1 m_{K_1}}\ast\Gamma_{\rho_2 m_{K_2}}\right]\cap\Gamma_{\rho m_{K_1K_2}}
=\{z\del_g\ast (\rho m_{K_1K_2}):z\in\Tee,g\in\langle 
H_1 H_2\rangle\}
\]
where $H_1=G_{K_1,\rho_1}\cap G_{K_1K_2,\rho}$ and 
$H_2=G_{K_2,\rho_2}\cap G_{K_1K_2,\rho}$.
\end{proposition}

\begin{proof}
Let us record some observations about contractive idempotents.  First we have that
$\supp(\rho m_K)=K$.  If $g$ in $G$ and $z$ in $\Tee$ are
is such that $z\del_g\ast(\rho m_K)=\rho' m_{K'}$, then $gK=\supp(z\del_g\ast(\rho m_K))
=\supp(\rho' m_{K'})=K'$, so $K=K'$ and $g\in K$.

To see the inclusion of the first set into the second, let $g_1\in G_{K_1,\rho_1}$, 
$g_2\in G_{K_2,\rho_2}$.  Then we compute
\[
\del_{g_1}\ast(\rho_1 m_{K_1})\ast \del_{g_2}\ast(\rho_2 m_{K_2})
=\del_{g_1}\ast(\rho m_{K_1K_2})\ast \del_{g_2}
=\del_{g_1g_2}\ast\del_{g_2^{-1}}\ast(\rho m_{K_1K_2})\ast\del_{g_2}
\]
where $\del_{g_2^{-1}}\ast(\rho m_{K_1K_2})\ast\del_{g_2}$ is a contractive idempotent.  If 
we assume that there is $g$ in $G_{K_1K_2,\rho}$ and $z$ in $\Tee$ for which
\[
\del_{g_1g_2}\ast\del_{g_2^{-1}}\ast(\rho m_{K_1K_2})\ast\del_{g_2}=z\del_g\ast(\rho m_{K_1K_2}).
\]
then it follows from the argument in the paragraph above that $g^{-1}g_1g_2\in K_1K_2$.
Hence $g^{-1}g_1\in K_1K_2\subseteq G_{K_1K_2,\rho}$ so $g_1\in H_1$. 
Also, as $g\in N_G(K_1K_2)$,  we have
$g_2\in K_1K_2 g\subseteq G_{K_1K_2,\rho}$, and we obtain that 
$g_2\in H_2$.   By Remark \ref{rem:stokke} (ii), any non-zero limit of products of elements of 
$\{z\del_g\ast (\rho m_{K_1K_2}):z\in\Tee,g\in\langle H_1 H_2\rangle\}$ remains in that set.

To see the reverse inclusion, we let $g_1\in H_1$ and $g_2\in H_2$ and we observe that 
\[
\del_{g_1g_2}\ast(\rho m_{K_1K_2})=\del_{g_1}\ast(\rho m_{K_1K_2})\ast\del_{g_2}
=\del_{g_1}\ast(\rho_1 m_{K_1})\ast \del_{g_2}\ast(\rho_2 m_{K_2}).
\]
We use Remark \ref{rem:stokke} (i), to see that non zero limits of products of such elements
remain in $\Gamma_{\rho m_{K_1K_2}}$.

Either argument above can be easily redone, multiplied by elements of $\Tee$.
\end{proof}

\begin{example} {\bf (i)} In the notation  above, suppose that $G_{K_1,\rho_1}=G$.
The happens, for example, if $K_1$ is in the centre of $G$.  Indeed, then $\ker\rho_1$
is in the centre of $G$, and $K_1/\ker\rho_1$ is in the centre of $G/\ker\rho_1$.
Then, in the assumption of Proposition \ref{prop:groups}, we have $G_{K_1,\rho_1}\cap
G_{K_1K_2,\rho}=G_{K_1K_2,\rho}$ and hence
\[
\left[\Gamma_{\rho_1 m_{K_1}}\ast\Gamma_{\rho_2 m_{K_2}}\right]\cap\Gamma_{\rho m_{K_1K_2}}
=\Gamma_{\rho m_{K_1K_2}}.
\]

{\bf (ii)}  In the notation above, we always have that $K_1\subseteq G_{K_1,\rho_1}$
and $K_2\subseteq G_{K_2,\rho_2}$.
Hence if $G=K_1K_2$, then by Proposition \ref{prop:groups}, we have
\[
\left[\Gamma_{\rho|_{K_1} m_{K_1}}\ast\Gamma_{\rho|_{K_2} m_{K_2}}\right]
\cap\Gamma_{\rho m_G}=\Gamma_{\rho m_G}
\]
for any $\rho\in\what{G}_1$.
This works even for ``non-trivial" matched pairs in the sense of Example \ref{ex:rider} (ii).

{\bf (iii)}  Let $T$ be any non-trivial compact abelian group, $\sig$  be given on $T\times T$ by 
$\sig(t_1,t_2)=(t_2,t_1)$ and $G=(T\times T)\rtimes \{\id,\sig\}$.  Let $\rho_1,\rho_2\in\what{T}$
(dual group of $T$) so $\rho_1\times\rho_2\in\what{T\times T}$.  
Then $N_G(T\times \{e\})=T\times T$ is abelian and hence
it is easy to follow the definition to see $G_{T\times\{e\},\rho_1}=T\times T$.
By symmetry, $G_{\{e\}\times T,\rho_2}=T\times T$, as well.  

On the other hand $N_G(T\times T)=G$, and 
$\sig(\ker\rho_1\times\rho_2)=\ker\rho_1\times\rho_2$, so $N_G(\ker\rho_1\times\rho_2)=G$.  Also
\[
G/\ker\rho_1\times\rho_2=[(T\times T)/\ker\rho_1\times\rho_2]\rtimes \{\id,\sig\}
\cong \rho_1\times\rho_2(T\times T)\times\{\id,\sig\}
\]
is abelian, i.e. $\sig$ acts trivially on the image $\rho_1\times\rho_2(T\times T)\cong(T\times T)/
\ker\rho_1\times\rho_2$.  Hence $G_{T\times T,\rho_1\times\rho_2}=G$.  Thus by
Proposition \ref{prop:groups} we have
\[
\left[\Gamma_{\rho_1 m_{T\times\{e\}}}\ast
\Gamma_{\rho_2 m_{\{e\}\times T}}\right]\cap\Gamma_{(\rho_1\times\rho_2) m_{T\times T}}
\subsetneq \Gamma_{(\rho_1\times\rho_2) m_{T\times T}}.
\]
\end{example}

We now consider some groups of measures considered in \cite{stokke}.
For $K$ in $\fK(G)$ and $\rho$ in $\what{K}_1$ let
\begin{equation}\label{eq:stokke}
\fM_{\rho m_K}=\{\nu\in\meas(G):\nu^*\ast\nu=\rho m_K=\nu\ast\nu^*\}.
\end{equation}
Notice that if $\nu\in\fM_{\rho m_K}$, then the operator
$\xi\mapsto\nu\ast\xi$ on $\mathrm{L}^2(G)$ is a partial isometry
with support and range projection $\xi\mapsto(\rho m_K)\ast\xi$.
Since the injection $\nu\mapsto(\xi\mapsto\nu\ast\xi)$ from $\meas(G)$
into bounded operators on $\mathrm{L}^2(G)$ is injective, it follows that
for $\nu$ in $\fM_{\rho m_K}$ that $\nu\ast(\rho m_K)=\nu=(\rho m_K)\ast\nu$.
We call $\fM_{\rho m_K}$ the {\it local unitary group} at $\rho m_K$.  
It is clear that $\Gamma_{\rho m_K}\subseteq\fM_{\rho m_K}$.

Our goal is to make a modest determination of the scope of $\fM_{m_K}$ for an
idempotent probability measure.  We begin with an analogue of a well-known
characterization of the stucture of the connected component of the invertible group of a Banach
algebra.  This lemma plays more of a role in motivating the methods below, than
in producing a result we shall use directly.

\begin{lemma}\label{lem:ccmdh}
Let $H$ be a locally compact group.  Then the connected component of the identity of
$\fM_{\del_e}$ in $\meas(H)$ is the group
\[
\fM_{\del_e,0}=\{\exp\lam_1\dots\exp\lam_n:\lam_1,\dots,\lam_n\in\meas(H)_{\mathrm{ska}},\,
n\in\En\}
\]
where $\meas(H)_{\mathrm{ska}}=\{\lam\in\meas(H):\lam^*=-\lam\}$, the real linear space
of skew-adjoint measures.
\end{lemma}

\begin{proof}
There exists norm-open neighbourhoods $B$ of $0$ and $U$ of $\del_e$, in $\meas(H)$,
on which $\exp:B\to U$ is a homeomorphism.   There is a logarithm defined on a neighbourhood
of $\del_e$, and analytic functional calculus shows these are mutually inverse.
We may suppose that $B$ is symmetric and  closed under the adjoint. 

If $\nu\in U\cap\fM_{\del_e}$, then there is some $\lam$ in $B$ for which $\nu=\exp\lam$, and we 
have $\exp(\lam^*)=\exp(\lam)^*=\nu^*=\nu^{-1}=\exp(-\lam)$, and hence $\lam^*=-\lam$,
by assumption on $B$.  If $\nu=\exp\lam_1\dots\exp\lam_n$, with 
$\lam_1,\dots,\lam_n\in\meas(H)_{\mathrm{ska}}$ and
$\nu'$ in $\fM_{\del_e}$ is so close to $\nu$ that $\nu^*\ast\nu'\in U$, then 
$\nu^*\ast\nu'=\exp\lam_{n+1}$ for some $\lam_{n+1}$ in $\meas(H)_{\mathrm{ska}}$.
The subgroup of all such products is hence open in $\fM_{\del_e}$ and clearly connected, thus
the connected component of $\del_e$.
\end{proof}

We say that a locally compact group $H$ is {\it Hermitian} if each element 
self-adjoint element of $\blone(H)$ has real spectrum.  
See \cite{palmer} for notes on the class of Hermitian groups.

\begin{proposition}\label{eq:bigger}
Let $K\in\fK(G)$.

{\bf (i)} If  $N_G(K)\supsetneq K$, then $\Gamma_{m_K}\subsetneq\fM_{m_K}$.

{\bf (ii)} If $N_G(K)/K$ is contains either a non-discrete closed abelian subgroup,
or a closed non-Hermitian subgroup, 
then the connected compoent of the identity $\fM_{m_K,0}$, is unbounded.
\end{proposition}

\begin{proof}
We let $H=N_G(K)/K$.  We notice, in passing, that $N_G(K)=G_{K,1}$.
The map $\vphi:\meas(N_G(K)/K)\to\meas(G)$ given for $u$ in $\fC_0(G)$ by
\[
\int_G u\,d\vphi(\nu)=\int_{N_G(K)/K}\int_K u(gk)\,dk\,dg=\int_{N_G(K)} u(g)\,dg.
\]
Since arbitrary elements of $\fC_0(N_G(K)/K)$ may be represented as
$gK\mapsto\int_K u(gk)\,dk$, as above, we see that $\vphi$ is injective, even isometric.
In particular
\[
\fM_{m_K}=\vphi(\fM_{\del_{e_H}})\text{ and }\Gamma_{m_K}
=\vphi(\Gamma_{\del_{e_H}})=\Tee\vphi(\del_H)
\]
where $\del_H=\{\del_h:h\in H\}$.

(i)  To see that the inclusion $\Gamma_{m_K}\subseteq\fM_{m_K}$
is proper, it suffices to see that $\Gamma_{\del_{e_H}}$, 
is a proper subgroup of $\fM_{\del_{e_H}}$.  Since
$H$ contains at least two elements, the real dimension of $\meas(H)_{\mathrm{ska}}$
is at least $2$.  Since $\exp$ is analytic and a homeomorphism on a neighbourhood $\wtil{B}$ of 
$0$ in $\meas(H)_{\mathrm{ska}}$, $\fM_{\del_{e_H}}$ contains a manifold or real dimension at 
least $2$.  But since $\del_H$ is norm discrete, we can pick $\wtil{B}$ small
enough so that $\exp(\wtil{B})\cap\Gamma_{\del_{e_H}}\subset\Tee\del_{e_H}$.
Hence $\exp(\wtil{B})\not\subset\Gamma_{\del_{e_H}}$.

(ii) If there exists $\nu=\nu^*$ in $\meas(H)$ with non-real spectrum, then
the one-parameter subgroup $\{\exp(it\nu)\}_{t\in\Ree}$ is unbounded and a subgroup of
$\fM_{\del_{e_H}}$. The Wiener-Pitt phenomenon
shows that if $H$ contains a closed non-discrete abelian subgroup $A$, 
then such a $\nu$ exists.  Indeed, if $\nu=\nu^*$ in $\meas(A)\subseteq\meas(H)$, 
then the Fourier-Steiltjes
transform satisfies $\hat{\nu}=\what{\nu^*}=\wbar{\hat{\nu}}$, 
and we appeal to Section 6.4  in \cite{rudin}.
If $H$ contains a closed non-Hermitian subgroup, we can choose
$\nu$ to be absolutely continuous with respect to Haar measure.
\end{proof}

It is not clear whether or not $\fM_{m_K}$ is always locally compact with respect to the weak*
topology.

\begin{remark} {\bf (i)} The proof of (i) above tells us that if $N_G(K)/K$ is infinite, then 
$\fM_{m_K}$ contains manifolds of arbitrarily high dimension.  
Thus we see that $\fM_{m_K}$ is not Lie, in this case.

{\bf (ii)} If $N_G(K)$ is compact, and hence so too is $H=N_G(K)/K$
with dual object $\what{H}$, then
$\fM_{m_K}\cong\fM_{\del_{e_H}}$ is a subgroup of the product of unitary groups
$\prod_{\pi\in\what{H}}\mathrm{U}(d_\pi)$, containing the dense restricted product subgroup, 
consisting of all elements which are $I_{d_\pi}$ for all but finitely many indices $\pi$.
Indeed, $\nu\mapsto(\pi(\nu))_{\pi\in\what{H}}:\meas(H)\to\ell^\infty\text{-}\oplus_{\pi\in\what{H}}
 M_{d_\pi}(\Cee)$ (notation as in (\ref{eq:pinu}))
injects $\fM_{\del_{e_H}}$ into the product group.  Furthermore, 
consider $u$ in $\prod_{\pi\in\what{H}}\mathrm{U}(d_\pi)$ where $u_{\pi}=I_{d_\pi}$
for all but $\pi_1,\dots,\pi_n$ in $\what{H}$, and $u_{\pi_k}=[u_{ij,k}]$ in
$\mathrm{U}(d_{\pi_k})$ for $k=1,\dots,n$.  The element of $\meas(H)$ given by
\[
\nu_u=\del_e+\sum_{k=1}^nd_{\pi_k}
\left(\sum_{i,j=1}^{d_{\pi_k}}u_{ij,k}\pi_{k,ij}-\sum_{j=1}^{d_{\pi_k}}\pi_{k,jj}\right)m_H
\]
where each set $\{\pi_{k,ij}\}_{i,j=1}^{d_{\pi_k}}$ are matrix coefficients of 
$\pi_k$ with respect to an orthonormal basis
for the space on which it acts, satisfies, with respect to the same basis, $\pi(\nu_u)=u_\pi$.  
Notice that $\nu_u\ast\nu_{u'}=\nu_{uu'}$ and $\nu_u^*=\nu_{u^*}$.
\end{remark}

{\bf Acknowledgement.} The author was supported by NSERC Grant 312515-2010.
The author is grateful to M. Neufang for pointing him towards the concept of a matched pair,
and to M. Neufang, P. Salmi and A. Skalski for helping him discover a mistake in an earlier version
of Theorem \ref{theo:contcommute}.  The author is also grateful to R. Stokke, for
stimulating discussion of, and pointing out some errors in, the last section.

\bibliographystyle{amsplain}

\end{document}